\documentclass[a4paper]{amsart}
\usepackage[new]{old-arrows}
\usepackage[T1]{fontenc}
\usepackage[english]{babel}
\usepackage{kpfonts}

\usepackage{amsmath,amssymb,amsthm,enumerate,xspace}
\usepackage{comment}
\usepackage{accents}
\usepackage[colorlinks=true,citecolor=blue,linkcolor=blue]{hyperref}
\usepackage[all]{xy}
\usepackage{tikz}
\usepackage{tikz-cd}
\usepackage{mathrsfs}
\usepackage{cleveref}
\usepackage{graphicx} 

\numberwithin{equation}{section}

\newtheorem{thm}{Theorem}[section]
\newtheorem{cor}[thm]{Corollary}

\newtheorem{lem}[thm]{Lemma}
\newtheorem{conj}[thm]{Conjecture}
\newtheorem*{conj*}{Conjecture}
\newtheorem{quest}[thm]{Question}

\newtheorem*{openproblem*}{Problem}
\newtheorem*{quest*}{Question}
\newtheorem*{problem*}{Problem}
\newtheorem*{claim*}{Claim}

\theoremstyle{definition}
\newtheorem{defn}[thm]{Definition}

\theoremstyle{remark}
\newtheorem{rem}[thm]{Remark}

\newcommand{\bF}{\mathbb{F}}

\newcommand{\bQ}{\mathbb{Q}}
\newcommand{\bR}{\mathbb{R}}
\newcommand{\bS}{\mathbb{S}}

\newcommand{\bZ}{\mathbb{Z}}
\newcommand{\Exterior}{\mathchoice{{\textstyle\bigwedge}}%
    {{\bigwedge}}%
    {{\textstyle\wedge}}%
    {{\scriptstyle\wedge}}}

\newcommand\Diff{\mathrm{Diff}}

\newcommand\BdDiff{\mathrm{BDiff}^{\delta}}

\newcommand{\hcoker}{/\!\!/}

\makeatother
\addtocontents{toc}{\protect\setcounter{tocdepth}{1}}
\makeatletter
\let\c@equation\c@thm
\makeatother
\numberwithin{equation}{section}
\title[]{PL homeomorphisms of surfaces and codimension $2$ PL foliations}

\author{Sam Nariman}
\address{Department of Mathematics\\
  Purdue University\\
150 N. University Street\\
West Lafayette, IN 47907-2067\\
}
\email{snariman@purdue.edu}

\dedicatory{To the memory of Andr\' e Haefliger}
\subjclass{57R30, 57R32, 58D05, 55R40, 57Q60}
\keywords{Haefliger classifying spaces, Piecewise Linear homeomorphisms, Perfectness and Simplicity of groups, MMM classes}
\begin{document}
\begin{abstract}
Haefliger--Thurston's conjecture predicts that Haefliger's classifying space for  $C^r$-foliations of codimension $n$ whose normal bundles are trivial is $2n$-connected. In this paper, we confirm this conjecture for PL foliations of codimension $2$. Using this, we use a version of Mather-Thurston's theorem for PL homeomorphisms due to the author to derive new homological properties for PL surface homeomorphisms. In particular, we answer the question of Epstein in dimension $2$ and prove the simplicity of the identity component of PL surface homeomorphisms.

\end{abstract}
\maketitle
\section{Introduction}
Haefliger defined the notion of $C^r$-Haefliger structures on manifolds which are more flexible than $C^r$-foliations to be able to construct a classifying space $\mathrm{B}\Gamma^{r}_n$ for them (\cite{haefliger1971homotopy, bott1972lectures}). This space is the classifying space of the \' etale groupoid $\Gamma_n$ of germs of local  $C^r$-diffeomorphisms of $\bR^n$. For $r>0$, there is a  map 
\[
\nu\colon \mathrm{B}\Gamma^{r}_n\to \mathrm{BGL}_n(\bR),
\]
which classifies the normal bundle to the $C^r$-Haefliger structures and the homotopy fiber of $\nu$ is denoted by $\overline{\mathrm{B}\Gamma}^{r}_n$. 

The work of Haefliger  (\cite{haefliger1971homotopy}) and Thurston's h-principle theorems (\cite{MR0370619,MR0425985}) say that if the normal bundle of a $C^r$-Haefliger structure $\gamma$ on a manifold $M$ can be embedded into the tangent bundle $TM$, then there is a genuine foliation in the homotopy class of $\gamma$. Hence, in principle, the classification of foliations on $M$ is translated into the homotopy type of the mysterious space $\overline{\mathrm{B}\Gamma}^{r}_n$. Haefliger proved that $\overline{\mathrm{B}\Gamma}^{r}_n$ is $n$-connected and Thurston proved (\cite{thurston1974foliations})  that the identity component of the smooth diffeomorphism group of any compact manifold is a simple group and used it to show that  $\overline{\mathrm{B}\Gamma}^{\infty}_n$ is $(n+1)$-connected and shortly after, Mather (\cite[Section 7]{MR0356129}) proved the same statement for  $\overline{\mathrm{B}\Gamma}^{r}_n$ for all regularities $r$ except  $r=n+1$. 

The theory of differentiable cohomology for groupoids developed by Haefliger made him speculate that $\overline{\mathrm{B}\Gamma}^{r}_n$ might be $2n$-connected and Thurston also stated  (\cite{thurston1974foliations}) this range of connectivity for $\overline{\mathrm{B}\Gamma}^r_n$ as a conjecture.
\begin{conj}[Haefliger-Thurston]\label{HT0}
The space $\overline{\mathrm{B}\Gamma}^r_n$ is $2n$-connected. 
\end{conj}
A geometric consequence of this conjecture (\cite{MR0370619,MR0425985}) is that any subbundle of the tangent bundle of a smooth $M$ whose dimension is at most $(\text{dim}(M)+1)/2$ is $C^r$-integrable up to homotopy i.e. one can change it up to homotopy to become the tangent field of a $C^r$-foliation on $M$. As a consequence of Mather's acyclicity result (\cite{MR0288777}) and McDuff's theorem (\cite{mcduff1980homology}), we know that $\overline{\mathrm{B}\Gamma}^0_n$ is contractible and it is a consequence of the remarkable theorem of Tsuboi (\cite{tsuboi1989foliated}) that $\overline{\mathrm{B}\Gamma}^1_n$ is also contractible. But for regularity $r>1$, because of the existence and nontriviality of Godbillon-Vey invariants, it is known that $\overline{\mathrm{B}\Gamma}^r_n$ is not $(2n+1)$-connected.

In this paper, we consider the piecewise linear (PL for short) category instead and we prove the analog of this conjecture for PL foliations of codimension $2$. To formulate the conjecture in this category, we shall first define $\overline{\mathrm{B}\Gamma}_n^{\text{PL}}$. Let $\Gamma_n^{\text{PL}}$ be the \' etale groupoid of germs of local orientation preserving PL homeomorphisms of $\bR^n$. The classifying space $\mathrm{B}\Gamma_n^{\text{PL}}$ classifies codimension $n$ PL Haefliger structures that are co-oriented up to concordance (\cite[Section 2]{haefliger1971homotopy}). Another perspective is that it classifies foliated PL microbundle of dimension $n$ (\cite[Page 188, Proposition]{haefliger1970feuilletages} or \cite[Section 4]{tsuboi2009classifying}). On the other hand, Kuiper--Lashof's theorem implies that oriented PL microbundles of dimension $n$ are classified by $\mathrm{B}\text{PL}^+(\bR^n)$ where $\text{PL}^+(\bR^n)$ is the realization of the simplicial group of orientation preserving PL homeomorphisms of $\bR^n$ (\cite{kuiper1966microbundles}). So forgetting the germ of the foliation near the zero section of microbundles induces a map
\[
\nu\colon \mathrm{B}\Gamma_n^{\text{PL}}\to \mathrm{B}\text{PL}^+(\bR^n).
\]
 This map classifies the normal microbundle of the PL Haefliger structures. Let $\overline{\mathrm{B}\Gamma}_n^{\text{PL}}$ be the homotopy fiber of the above map $\nu$. Haefliger's argument (\cite[Section 6]{haefliger1971homotopy}) implies that $\overline{\mathrm{B}\Gamma}_n^{\text{PL}}$ is $(n-1)$-connected. He proved in (\cite[Theorem 3]{haefliger1970feuilletages}) that Phillips' submersion theorem in the smooth category implies that $\overline{\mathrm{B}\Gamma_n}^r$ is $n$-connected for $r>0$. Given that Phillips' submersion theorem also holds in the PL category (\cite{haefliger1964classification}), one could argue exactly similar to the smooth case to show that $\overline{\mathrm{B}\Gamma}_n^{\text{PL}}$ is in fact $n$-connected. The analog of Haefliger--Thurston's conjecture in the PL category is that the classifying space $\overline{\mathrm{B}\Gamma}_n^{\text{PL}}$ is $2n$-connected. Unlike the smooth case, it is not even known whether $\overline{\mathrm{B}\Gamma}_n^{\text{PL}}$ is $(n+1)$-connected for all $n$. Our main theorem is about the connectivity of this space for $n=2$.
\begin{thm}\label{PL}
The space $\overline{\mathrm{B}\Gamma}_2^{\textnormal{\text{PL}}}$ is  $4$-connected.
\end{thm}
As a consequence of this theorem, we prove new homological properties of PL surface homeomorphisms.
\subsection{Applications} Let $M$ be a compact connected PL $n$-manifold possibly with non-empty boundary. Let $\text{PL}(M, \text{rel }\partial)$ be the group of PL homeomorphisms of $M$ which agree with the identity on an open neighborhood of the boundary and let $\text{PL}_{0}(M, \text{rel }\partial)$ be the identity component of $\text{PL}(M, \text{rel }\partial)$. Epstein (\cite[Theorem 3.1]{MR267589}) considered the abstract group $G={\textnormal{\text{PL}}}^{\delta}_{0}(M, \text{rel }\partial)$ and showed that the commutator subgroup $[G,G]$ is a simple group. Hence, to prove that $G$ is simple, it is enough to show that it is perfect.  He then proved that $\text{PL}([0,1], \text{rel }\partial)$ and $\text{PL}_0(S^1)$ are perfect by observing that in dimension one,  PL homeomorphisms are generated by certain ``typical elements" and those typical elements can be easily written as commutators. To generalize his argument to higher dimensions, he suggested the following approach (\cite[Page 173]{MR267589}).
\begin{defn}Let $B$ be a ball in $\bR^n$. It is PL homeomorphic to $S^{n-2}\star [0,1]$, the join of $S^{n-2}$ with $[0,1]$. Note that for PL manifolds $M$ and $N$, a PL homeomorphism of $N$ extends naturally to a PL homeomorphism of the join $M\star N$.  A {\it glide} homeomorphism of the ball $B$ is a PL homeomorphism that is induced by the extension of a compactly supported PL homeomorphism of $(0,1)$ to a PL homeomorphism of $S^{n-2}\star [0,1]$. For a PL $n$-manifold $M$, a glide homeomorphism $h\colon M\to M$ is the extension by the identity of a glide homeomorphism supported in a PL embedded ball $B\hookrightarrow M$. 
\end{defn}
\begin{quest*}[Epstein]
Is ${\textnormal{\text{PL}}}^{\delta}_{0}(M, \text{rel }\partial)$ generated by glide homeomorphisms?
\end{quest*}
The affirmative answer to this question implies that  ${\textnormal{\text{PL}}}^{\delta}_{0}(M, \text{rel }\partial)$ is simple and conversely since the group generated by glide homeomorphisms is a normal subgroup of ${\textnormal{\text{PL}}}^{\delta}_{0}(M, \text{rel }\partial)$ if the group ${\textnormal{\text{PL}}}^{\delta}_{0}(M, \text{rel }\partial)$ is simple, it is generated by glide homeomorphisms. 

\begin{thm}
Let $\Sigma$ be an oriented compact surface possibly with a boundary. Then the group ${\textnormal{\text{PL}}}^{\delta}_0(\Sigma,\text{rel }\partial)$ is simple. 
\end{thm}

The {\it simplicity} of  ${\textnormal{\text{PL}}}^{\delta}_{0}(M, \text{rel }\partial)$ in all dimensions is still open. We use \Cref{PL} and the version of {\it Mather-Thurston's theorem} that the author proved in (\cite[Section 5]{nariman2020thurston}) for PL homeomorphisms to prove the perfectness of this group in dimension $2$ and as a consequence, we answer affirmatively Epstein's question in dimension $2$. 

Note that this line of argument is the opposite of Thurston's point of view in the smooth category where he first proved the perfectness of the identity component of smooth diffeomorphism groups to improve the connectivity of the space $\overline{\mathrm{B}\Gamma}^{\infty}_n$. Our argument gives a homotopy theoretic proof of the perfectness of  ${\textnormal{\text{PL}}}^{\delta}_{0}(M, \text{rel }\partial)$ for a compact surface $M$. It would be still interesting to find a direct algebraic proof and study the commutator length of PL surface homeomorphisms.

Recall that the perfectness of a group is equivalent to the vanishing of its first group homology. We in fact determine the group homology of PL surface homeomorphisms up to degree $2$.  By Hauptvermutung in dimension $2$, any two PL structures on a surface $\Sigma$ are PL homeomorphic (\cite[Chapter 8]{MR0488059}). Hence, the homotopy type of $\text{PL}(\Sigma,\text{rel }\partial)$  does not depend on the choice of the PL structure on the surface $\Sigma$. Moreover, we have the weak equivalence $\text{PL}(\Sigma,\text{rel }\partial)\simeq \text{Homeo}(\Sigma,\text{rel }\partial) $ (\cite[Page 8]{MR652596}).
\begin{thm}\label{perfectness}
Let $\Sigma$ be a compact orientable surface, then the natural map 
\[
\mathrm{B}{\textnormal{\text{PL}}}^{\delta}_0(\Sigma,\text{rel }\partial)\to \mathrm{B}{\textnormal{\text{PL}}}_0(\Sigma,\text{rel }\partial)
\]
induces an isomorphism on $H_*(-;\bZ)$ for $*\leq 2$ and surjection on $*=3$. 
\end{thm}
Since $\mathrm{B}{\textnormal{\text{PL}}}_0(\Sigma,\text{rel }\partial)$ is simply connected, this theorem implies that ${\textnormal{\text{PL}}}^{\delta}_0(\Sigma,\text{rel }\partial)$ is a perfect group. Therefore, \Cref{perfectness} proves the perfectness of PL homeomorphisms of surfaces without following Epstein's strategy through glide homeomorphisms. On the other hand, by Epstein (\cite[Theorem 3.1]{MR267589}), the perfectness for the group ${\textnormal{\text{PL}}}^{\delta}_0(\Sigma,\text{rel }\partial)$ implies that it is also {\it simple}. Given that the group generated by glides is a normal subgroup, we can also answer Epstein's question about glide homeomorphisms.

In fact, the homotopy type of the topological group ${\textnormal{\text{PL}}}_0(\Sigma,\text{rel }\partial)$ is completely determined so the second group homology of ${\textnormal{\text{PL}}}^{\delta}_0(\Sigma,\text{rel }\partial)$ can also be determined. 

Our second application is about the invariants of flat surface bundles with transverse PL structures. We first show that the powers of the universal Euler class are all nontrivial in $H^{*}(\mathrm{B}\Gamma_2^{\textnormal{\text{PL}}};\bZ)$ and we use it to prove the following.
\begin{thm}\label{MW}
Let $\Sigma$ be a compact orientable surface, then the map
\[
H^*(\mathrm{B}{\textnormal{\text{PL}}}(\Sigma,\text{rel }\partial);\bQ)\to H^*(\mathrm{B}{\textnormal{\text{PL}}}^{\delta}(\Sigma,\text{rel }\partial);\bQ),
\]
induces an injection when $*\leq (2g(\Sigma)-2)/3$ where $g(\Sigma)$ is the genus of the surface $\Sigma$.
\end{thm} 
As a consequence of Madsen-Weiss' theorem (\cite{madsen2007stable}), $H^*(\mathrm{B}{\textnormal{\text{PL}}}(\Sigma,\text{rel }\partial);\bQ)$ is isomorphic to the polynomial ring $\bQ[\kappa_1,\kappa_2,\dots]$ in the stable range i.e. $*\leq (2g(\Sigma)-2)/3$.  Here $\kappa_i$ are certain characteristic classes of surface bundles known as $i$-th MMM classes whose degree is $2i$. 
\begin{cor}
$\kappa_i$ are all non-trivial in $H^*(\mathrm{B}{\textnormal{\text{PL}}}^{\delta}(\Sigma,\text{rel }\partial);\bQ)$ as long as $i\leq (2g(\Sigma)-2)/6$.
\end{cor}
This is in contrast with the case of smooth diffeomorphisms. It is known by the Bott vanishing theorem (\cite[Theorem 8.1]{morita1987characteristic}) that $\kappa_i$ vanishes in $H^*(\BdDiff(\Sigma,\text{rel }\partial);\bQ)$ for all $i>2$ and Kotschick-Morita (\cite{kotschick2005signatures}) proved that $\kappa_1$ is non-trivial in $H^2(\BdDiff(\Sigma,\text{rel }\partial);\bQ)$ as long as $g(\Sigma)\geq 3$. However, it is still open (\cite{kotschick2005signatures}) whether $\kappa_2$ is non-trivial in $H^4(\BdDiff(\Sigma,\text{rel }\partial);\bQ)$.
\subsection*{Acknowledgement} The author was partially supported by NSF DMS-2113828, NSF CAREER Grant DMS-2239106 and Simons Foundation (855209, SN). The author would like to thank Gael Meigniez for the discussion about Haefliger--Thurston's conjecture and Alexander Kupers for bringing up Suslin's calculations in \cite{Suslin} and his comments on the proof of  \Cref{X}. The author also thanks the referees for very helpful comments and suggestions to improve the readability, in particular, one of them suggested \Cref{R_A}. 

\section{The curious case of PL foliations} \label{PLfol}In this section, we are mainly concerned with codimension $2$, PL foliations that are co-oriented (i.e. their normal bundles are oriented). The Haefliger classifying space for these structures $\mathrm{B}\Gamma_2^{PL}$ is the (fat) geometric realization of the nerve of the \' etale groupoid $\Gamma_2^{PL}$ whose space of objects is $\bR^2$ with the usual topology and the space of morphisms are germs of orientation preserving PL homeomorphisms of $\bR^2$ with the sheaf topology (\cite[Section 1]{mather2011homology}). The main inputs to prove \Cref{PL} are Greenberg's inductive model for the classifying space PL foliations (\cite{MR1200422}) and Suslin's work (\cite{Suslin}) on low degree K-groups of real numbers. And then we use our variant of Mather-Thurston's theorem (\cite{nariman2020thurston}) for PL homeomorphisms to relate the connectivity of $\overline{\mathrm{B}\Gamma}_2^{\textnormal{\text{PL}}}$ to the homology of PL surface homeomorphisms. 

 We first recall Greenberg's recursive construction for such classifying spaces (\cite[Theorem 3.2 (c)]{MR1200422})  in the case that we are interested in. 
 \subsection{Greenberg's construction and the connectivity of $\overline{\mathrm{B}\Gamma}_2^{\textnormal{\text{PL}}}$} Let $A$ be the subgroup of $\text{GL}_2^+(\bR)$ consisting of matrices of the form
\begin{equation*}
M=
\begin{bmatrix}
a& b \\
0 & d 
\end{bmatrix}
\end{equation*}
where $a$ and $d$ are positive reals. Let $\epsilon\colon A\to \bR^+$ be the homomorphism $\epsilon(M)=a$. Let $R_A$ be the following  homotopy pushout
\begin{equation}\label{1}
\begin{gathered}
 \begin{tikzpicture}[node distance=1.5cm, auto]
  \node (A) {$ \mathrm{B}(A\times_{\bR^+} A)^{\delta}$};
  \node (B) [right of=A, node distance=3cm] {$\mathrm{B}A^{\delta}$};
  \node (C) [below of=A]{$\mathrm{B}A^{\delta}$};
  \node (D) [right of=C, node distance=3cm]{$R_A,$};
  \draw [->] (A) to node {$p_1$}(B);
    \draw [->] (A) to node {$p_2$}(C);
  \draw [->] (C) to node {$$}(D);
  \draw [->] (B) to node {$$}(D);
\end{tikzpicture}
\end{gathered}
\end{equation}
where $p_i$ are induced by the projection to the $i$-th factor and $A\times_{\bR^+} A$ is the fiber products of $A$ over $\bR^+$ using the map $\epsilon$. Let $\widetilde{\text{GL}_2^+}(\bR)$ be the universal cover of $\text{GL}_2^+(\bR)$. Note that the inclusion of $A$ into $\text{GL}_2^+(\bR)$ lifts to the universal cover $\widetilde{\text{GL}_2^+}(\bR)$. Let the map $\alpha\colon \mathrm{B}A^{\delta}\to R_A$ be induced by the diagonal embedding $A\to A\times_{\bR^+} A$ and then composing with $\mathrm{B}(A\times_{\bR^+} A)^{\delta}\to R_A$. We let $X$ be the homotopy pushout of the following diagram.

\begin{equation}\label{2}
\begin{gathered}
 \begin{tikzpicture}[node distance=1.5cm, auto]
  \node (A) {$ \mathrm{B}A^{\delta}$};
  \node (B) [right of=A, node distance=3cm] {$R_A$};
  \node (C) [below of=A]{$\mathrm{B}\widetilde{\text{GL}_2^+}(\bR)^{\delta}$};
  \node (D) [right of=C, node distance=3cm]{$X.$};
  \draw [->] (A) to node {$$}(B);
    \draw [->] (A) to node {$$}(C);
  \draw [->] (C) to node {$$}(D);
  \draw [->] (B) to node {$$}(D);
\end{tikzpicture}
\end{gathered}
\end{equation}

Finally, let $LX$ be the space of continuous free loops in $X$ and let $LX\hcoker S^1$ be the homotopy quotient of the circle action on $LX$. We define $rX$ to be the homotopy pushout 
\begin{equation}\label{3}
\begin{gathered}
 \begin{tikzpicture}[node distance=1.5cm, auto]
  \node (A) {$ LX$};
  \node (B) [right of=A, node distance=3cm] {$X$};
  \node (C) [below of=A]{$LX\hcoker S^1$};
  \node (D) [right of=C, node distance=3cm]{$rX,$};
  \draw [->] (A) to node {$\text{ev}$}(B);
    \draw [->] (A) to node {$j$}(C);
  \draw [->] (C) to node {$q$}(D);
  \draw [->] (B) to node {$$}(D);
\end{tikzpicture}
\end{gathered}
\end{equation}
where $\text{ev}\colon LX\to X$ is the evaluation of the loops at their base point and $j$ is the inclusion of the fiber in the Borel fibration $LX\to LX\hcoker S^1\to \mathrm{B}S^1$. Greenberg's theorem (\cite[Theorem 3.2 (c)]{MR1200422})\footnote{He also explains his statement in his introduction but the statement in the introduction is missing a diagram and only describes the space $X$ instead of $rX$.} says that $rX\simeq \mathrm{B}\Gamma_2^{PL}$.  Recall $\overline{\mathrm{B}\Gamma}_2^{PL}$ is the homotopy fiber of the map
\[
\nu\colon\mathrm{B}\Gamma_2^{PL}\to \mathrm{B}\text{PL}^+(\bR^2)\simeq \mathrm{B}S^1,
\]
and it was already known as we mentioned in the introduction that $\overline{\mathrm{B}\Gamma}_2^{PL}$ is at least $2$-connected. So to prove \Cref{PL}, it is enough to show that the map $\nu$ induces a homology isomorphism up to and including degree $4$. To do this, we shall calculate the homology of $rX$ using the Mayer-Vietoris sequence for the homotopy pushout square \ref{3}. But we first need to prove the following key lemma about the homotopy type of $X$. 
\begin{thm}\label{X}
The space $X$ is  $2$-connected.
\end{thm}
The fact that $X$ is simply connected was already observed by Greenberg (\cite[Proof of Corollary 2.6]{MR1200422}). This can also be seen using Van Kampen's theorem to compute the fundamental group. The map $\mathrm{B}\epsilon\colon \mathrm{B}A^{\delta}\to \mathrm{B}\bR^{+,\delta}$ induces a map $h\colon R_A\to \mathrm{B}\bR^{+,\delta}$. Using Van Kampen's theorem, one can easily see that $h$ induces an isomorphism $\pi_1(R_A)\xrightarrow{\cong} \bR^+$. So $\pi_1(X)$ is isomorphic to the quotient of $\widetilde{\text{GL}_2^+}(\bR)^{\delta}$ by the smallest normal subgroup generated by the image of $\text{ker}(\epsilon)$. A priori, $\text{ker}(\epsilon)$ lies in $\text{GL}_2^+(\bR)^{\delta}$ and it is easy to see that it normally generates the entire group $\text{GL}_2^+(\bR)^{\delta}$. Therefore, its lift also normally generates $\widetilde{\text{GL}_2^+}(\bR)^{\delta}$. Hence, $\pi_1(X)$ is trivial. 

 One can use Milnor-Friedlander's conjecture for solvable Lie groups which was already proved in Milnor's original paper (\cite{milnor1983homology}) on this topic and Suslin's stability theorem (\cite{MR772065}) to show that $\pi_2(X)\otimes \bF_p=0$ for all prime $p$. But to prove the integral result that $\pi_2(X)=0$ we need to work a bit harder.
 
 \begin{rem} To see $\pi_2(X)\otimes \bF_p=0$, by the Hurewicz theorem, it is enough to show that the group $H_2(X;\bF_p)$ vanishes for all prime $p$. We shall first observe that $R_A$ is an $\bF_p$-acyclic space i.e. $H_*(R_A;\bF_p)=0$ for all $*>0$. We have the short exact sequence of groups
\[
\text{Aff}^+(\bR)^{\delta}\to A^{\delta}\to \bR^+.
\]
Therefore, the group $A^{\delta}$ is solvable. Similarly $(A\times_{\bR^+} A)^{\delta}$ is solvable. On the other hand, as topological groups, both $A$ and $A\times_{\bR^+} A$ are contractible. Hence, by Milnor's theorem (\cite[Lemma 3]{milnor1983homology}), the groups $A$ and $A\times_{\bR^+} A$ are $\bF_p$-acyclic for all prime $p$ and by applying the Mayer-Vietoris sequence to the pushout \ref{1}, we deduce that $R_A$ is also $\bF_p$-acyclic for all prime $p$. Now using the Mayer-Vietoris sequence with $\bF_p$ coefficients for the pushout \ref{2}, it is enough to show that $H_2(\mathrm{B}\widetilde{\text{GL}_2^+}(\bR)^{\delta};\bF_p)=0$. Recall that we know by Suslin's theorem (\cite{MR772065}) in general, the map
\[
\mathrm{B}\text{GL}_n^+(\bR)^{\delta}\to \mathrm{B}\text{GL}_n^+(\bR),
\]
induces an isomorphism on $H_*(-;\bF_p)$ for $*\leq n$. On the other hand,  we have a short exact sequence
\[
\bZ\to \widetilde{\text{GL}_2^+}(\bR)^{\delta}\to \text{GL}_2^+(\bR)^{\delta}.
\]
Therefore, by a spectral sequence argument, we deduce that  the map
\[
\mathrm{B}\widetilde{\text{GL}_2^+}(\bR)^{\delta}\to \mathrm{B}\widetilde{\text{GL}_2^+}(\bR),
\]
 induces an isomorphism on $H_*(-;\bF_p)$ for $*\leq 2$. But $\widetilde{\text{GL}_2^+}(\bR)$ is contractible which implies that $H_2(\mathrm{B}\widetilde{\text{GL}_2^+}(\bR)^{\delta};\bF_p)=0$.
\end{rem}
To prove \Cref{X}, we need some preliminary lemmas to do the calculations integrally.
\begin{lem}\label{retraction}
Let $D$ be the subgroup of diagonal matrices in $A$. Let $\iota \colon D\hookrightarrow A$ be the inclusion map of the subgroup of diagonal matrices.  The map $\iota$ has a left inverse $r\colon A\to D$
\[
r(
\begin{bmatrix}
a& b \\
0 & d 
\end{bmatrix})=\begin{bmatrix}
a& 0 \\
0 & d 
\end{bmatrix}.
\]
 The maps $\tilde{\iota}$ and $\tilde{r}$ 
\[
\mathrm{B}D^{\delta}\xrightarrow{\tilde{\iota}} \mathrm{B}A^{\delta}\xrightarrow{\tilde{r}} \mathrm{B}D^{\delta},
\]
induce homology isomorphisms. 
\end{lem}
\begin{proof}
There is a trick that apparently goes back to Quillen over rational coefficients (\cite[Lemma 4]{de1983acyclic}) and to Nesterenko-Suslin (\cite[Theorem 1.11]{nesterenko1990homology}) over integer coefficients that the map $\mathrm{B}\text{GL}_n^+(\bR)^\delta\to\mathrm{B}\text{Aff}^+(\bR^n)^{\delta}$ which also has a left inverse, induces a homology isomorphism. Taking $n=1$, we have a map between fibrations
\begin{equation}
\begin{gathered}
 \begin{tikzpicture}[node distance=1.5cm, auto]
  \node (A) {$ \mathrm{B} \bR^{+,\delta}$};
  \node (B) [right of=A, node distance=2cm] {$\mathrm{B}\text{Aff}^+(\bR^1)^{\delta}$};
  \node (C) [right of=B,  node distance=2cm]{$\mathrm{B} \bR^{+,\delta}$};
  \node (D) [below of=A, node distance=1.5cm]{$\mathrm{B}D^{\delta}$};
    \node (E) [below of=B, node distance=1.5cm]{$\mathrm{B}A^{\delta}$};
  \node (F) [below of=C, node distance=1.5cm]{$\mathrm{B}D^{\delta}$};
    \node (G) [below of=D, node distance=1.5cm]{$\mathrm{B} \bR^{+,\delta}$};
        \node (H) [below of=E, node distance=1.5cm]{$\mathrm{B} \bR^{+,\delta}$};
    \node (I) [below of=F, node distance=1.5cm]{$\mathrm{B} \bR^{+,\delta},$};

  \draw [->] (A) to node {$$}(B);
    \draw [->] (A) to node {$$}(D);
  \draw [->] (C) to node {$$}(F);
  \draw [->] (B) to node {$$}(C);
    \draw [->] (B) to node {$$}(E);
      \draw [->] (B) to node {$$}(C);
  \draw [->] (D) to node {$$}(G);
  \draw [->] (E) to node {$$}(H);
  \draw [->] (F) to node {$$}(I);
  \draw [->] (H) to node {$\cong$}(I);
  \draw [->] (G) to node {$\cong$}(H);

  \draw [->] (D) to node {$\tilde{\iota}$}(E);
    \draw [->] (E) to node {$\tilde{r}$}(F);

\end{tikzpicture}
\end{gathered}
\end{equation}
where the top horizontal maps induce homology isomorphisms and the bottom maps are the identity. Therefore, by the comparison of Serre spectral sequences, $\tilde{\iota}$ and $\tilde{r}$ also induce homology isomorphisms.
\end{proof}
Let $Y$ be the homotopy pushout of the diagram
\begin{equation}\label{Y}
\begin{gathered}
 \begin{tikzpicture}[node distance=1.5cm, auto]
  \node (A) {$ \mathrm{B}(D\times_{\bR^+} D)^{\delta}$};
  \node (B) [right of=A, node distance=3cm] {$\mathrm{B}D^{\delta}$};
  \node (C) [below of=A]{$\mathrm{B}D^{\delta}$};
  \node (D) [right of=C, node distance=3cm]{$Y.$};
  \draw [->] (A) to node {$p_1$}(B);
    \draw [->] (A) to node {$p_2$}(C);
  \draw [->] (C) to node {$$}(D);
  \draw [->] (B) to node {$$}(D);
\end{tikzpicture}
\end{gathered}
\end{equation}
Given that we have the inclusion $D\xrightarrow{\iota} A$ and its left inverse $r$, we have a natural map $\theta\colon Y\to R_A$. 
\begin{lem}\label{R_A}
The map $\theta\colon Y\to R_A$ admits a left inverse and it induces a homology isomorphism. The space $Y$ is homotopy equivalent to  $ \mathrm{B} \bR^{+,\delta} \times (\mathrm{B} \bR^{+,\delta}\star \mathrm{B} \bR^{+,\delta})\to R_A$ where $\star$ means the join of topological spaces.
\end{lem}
\begin{proof}
The inclusion $D\xrightarrow{\iota} A$ induces a map of homotopy pushout diagrams from (\ref{Y}) to (\ref{1}) and the left inverse $r$ induces a map diagrams from (\ref{1}) to (\ref{Y}) which gives the left inverse to $\theta$. Since the maps between corresponding terms induce homology isomorphisms by \Cref{retraction}, Mayer-Vietoris implies that $\theta$ induces a homology isomorphism. 

To see that $Y$ is homotopy equivalent to $ \mathrm{B} \bR^{+,\delta} \times(\mathrm{B} \bR^{+,\delta}\star \mathrm{B} \bR^{+,\delta})$, note that $D$ is isomorphic to $\bR^{+}\times \bR^{+}$. So $Y$ is homotopy equivalent to the homotopy pushout of 

\begin{equation}
\begin{gathered}
 \begin{tikzpicture}[node distance=1.5cm, auto]
  \node (A) {$\mathrm{B} \bR^{+,\delta}\times\mathrm{B} \bR^{+,\delta}\times \mathrm{B} \bR^{+,\delta}$};
  \node (B) [right of=A, node distance=4cm] {$\mathrm{B} \bR^{+,\delta}\times \mathrm{B} \bR^{+,\delta}$};
  \node (C) [below of=A]{$\mathrm{B} \bR^{+,\delta}\times \mathrm{B} \bR^{+,\delta}$};
  \node (D) [right of=C, node distance=4cm]{$Y,$};
  \draw [->] (A) to node {$p_{1,2}$}(B);
    \draw [->] (A) to node {$p_{1,3}$}(C);
  \draw [->] (C) to node {$$}(D);
  \draw [->] (B) to node {$$}(D);
\end{tikzpicture}
\end{gathered}
\end{equation}
where $p_{1,i}$ is the projection to the first and the $i$-th factor. Therefore, $Y$ is homotopy equivalent to $ \mathrm{B} \bR^{+,\delta} \times(\mathrm{B} \bR^{+,\delta}\star \mathrm{B} \bR^{+,\delta})$.
\end{proof}
 Recall that the map  $\alpha\colon \mathrm{B}A^{\delta}\to R_A$ is defined to be the composition of the diagonal embedding $\mathrm{B}A^{\delta}\to \mathrm{B}(A\times_{\bR^+} A)^{\delta}$ and $\mathrm{B}(A\times_{\bR^+} A)^{\delta}\to R_A$. Similarly, we obtain a map $\beta\colon \mathrm{B}D^{\delta}\to Y$. So we have a commutative diagram
 \begin{equation}
 \begin{gathered}
 \begin{tikzpicture}[node distance=1.5cm, auto]
  \node (A) {$ \mathrm{B}D^{\delta}$};
  \node (B) [right of=A, node distance=3cm] {$Y$};
  \node (C) [below of=A]{$\mathrm{B}A^{\delta}$};
  \node (D) [right of=C, node distance=3cm]{$R_A,$};
  \draw [->] (A) to node {$\beta$}(B);
    \draw [->] (A) to node {$$}(C);
  \draw [->] (C) to node {$\alpha$}(D);
  \draw [->] (B) to node {$$}(D);
\end{tikzpicture}
\end{gathered}
\end{equation}
 where the vertical maps induce homology isomorphisms. Note that the join $\mathrm{B} \bR^{+,\delta}\star \mathrm{B} \bR^{+,\delta}$ is $2$-connected and there is an isomorphism from $H_*(\mathrm{B} \bR^{+,\delta}; \bZ)$ to $\Exterior^*\bR^+$ which is an exterior product of $\bR^+$ over $\bZ$. By considering the commutative diagram
  \begin{equation}
 \begin{gathered}
 \begin{tikzpicture}[node distance=1.5cm, auto]
  \node (A) {$ H_*(\mathrm{B}D^{\delta}; \bZ)$};
  \node (B) [right of=A, node distance=5cm] {$H_*(Y; \bZ)$};
  \node (C) [below of=A]{$H_*(\mathrm{B} \bR^{+,\delta}\times \mathrm{B} \bR^{+,\delta}; \bZ)$};
  \node (D) [right of=C, node distance=5cm]{$H_*( \mathrm{B} \bR^{+,\delta} \times(\mathrm{B} \bR^{+,\delta}\star \mathrm{B} \bR^{+,\delta}); \bZ),$};
  \draw [->] (A) to node {$\beta_*$}(B);
    \draw [->] (A) to node {$\cong$}(C);
  \draw [->] (C) to node {$$}(D);
  \draw [->] (B) to node {$\cong$}(D);
\end{tikzpicture}
\end{gathered}
\end{equation}
 and using K{\" u}nneth's formula, it is easy to determine the kernel of the map $\beta_*$
in low homological degrees.   So we record the following corollary about $\ker \alpha_{*}$ in low homological degrees where 
 \[
 \alpha_*\colon H_*(\mathrm{B}A^{\delta}; \bZ)\to H_*(R_A; \bZ).
 \]
 \begin{cor}\label{iota_1}
 Let $t\colon \bR^+\to A$ be the map $t(a)=\begin{bmatrix}
1& 0 \\
0 & a 
\end{bmatrix}$ and let $s\colon \bR^+\to A$ be the map $s(a)=\begin{bmatrix}
a& 0 \\
0 & 1 
\end{bmatrix}$. 
\begin{enumerate}
 \item  The map $\iota\colon \bR^+\times \bR^+\cong D\to A$ induces a split surjection 
 \[
 \Exterior^2\bR^+\oplus (\bR^+\otimes \bR^+)\oplus \Exterior^2\bR^+\cong H_2(\mathrm{B}D^{\delta}; \bZ)\xrightarrow{\cong} H_2(\mathrm{B}A^{\delta}; \bZ)\to H_2(R_A; \bZ)\cong \Exterior^2\bR^+,
 \]
 which maps the first summand $ \Exterior^2\bR^+$ isomorphically to $H_2(R_A; \bZ)$. So $\ker \alpha_2$ is isomorphic to $(\bR^+\otimes \bR^+)\oplus \Exterior^2\bR^+$.
\item The map $t$ induces an isomorphism $\bR^+\cong H_1(\mathrm{B} \bR^{+,\delta}; \bZ)\to \ker \alpha_{1}$ and it also induces an injective map $\Exterior^2\bR^+\cong H_2(\mathrm{B} \bR^{+,\delta}; \bZ)\to \ker \alpha_{2}$ which is an isomorphism to the $\Exterior^2\bR^+$ summand of $\ker \alpha_{2}$ in the identification in previous item (1).
\item The following composition 
 \[
  \Exterior^2\bR^+\cong H_2(\mathrm{B}\bR^{+,\delta}; \bZ)\xrightarrow{s_2}H_2(\mathrm{B}A^{\delta}; \bZ)\to H_2(R_A; \bZ)\cong \Exterior^2\bR^+,
 \]
where the first map is induced by $s$ is an isomorphism.   

\end{enumerate}

 \end{cor}
 
 
 We also need part of Suslin's calculation (\cite[Theorem 2.1]{Suslin} and \cite[Chapter 6, section 5, Proof of Theorem 5.7]{weibel2013k}) of  $H_2(\mathrm{B}\text{GL}_2(\bR)^{\delta}; \bZ)$ to determine the image of
 \[
 H_2(\mathrm{B}A^{\delta}; \bZ)\to H_2(\mathrm{B}\widetilde{\text{GL}_2^+}(\bR)^{\delta}; \bZ).
 \] 
 
 To find $H_2(\mathrm{B}\widetilde{\text{GL}_2^+}(\bR)^{\delta}; \bZ)$, first note that there is an isomorphism $f\colon \text{SL}_2(\bR)\times \bR^{+}\xrightarrow{\cong} \text{GL}_2^+(\bR)$ where  $$f(A, a)=\begin{bmatrix}
a& 0 \\
0 & a 
\end{bmatrix}\cdot A.$$ 

This isomorphism can be lifted to give an isomorphism $\tilde{f}\colon  \widetilde{\text{SL}_2}(\bR)\times \bR^{+}\xrightarrow{\cong} \widetilde{\text{GL}_2^+}(\bR)$. On the other hand, the groups $\text{SL}_2(\bR)^{\delta} $ and $\widetilde{\text{SL}_2}(\bR)^{\delta}$ are perfect and it is known (\cite[Page 190-191]{MR722372}) that $
 H_2(\mathrm{B}\widetilde{\text{SL}_2}(\bR)^{\delta}; \bZ)\cong K_2(\bR),$ and we have a short exact sequence
 \[
0\to  H_2(\mathrm{B}\widetilde{\text{SL}_2}(\bR)^{\delta}; \bZ)\to  H_2(\mathrm{B}\text{SL}_2(\bR)^{\delta}; \bZ)\to \bZ\to 0.
\]
 Therefore, the K{\"u}nneth formula implies that  we have the isomorphism $H_2(\mathrm{B}\widetilde{\text{GL}_2^+}(\bR)^{\delta}; \bZ)\cong K_2(\bR)\oplus  \Exterior^2 \bR^{+}$ where $K_2(\bR)$ summand comes from the image of $H_2(\mathrm{B}\widetilde{\text{SL}_2}(\bR)^{\delta}; \bZ)\to H_2(\mathrm{B}\widetilde{\text{GL}_2^+}(\bR)^{\delta}; \bZ)$. Also the map
\[
u\colon H_2(\mathrm{B}\widetilde{\text{GL}_2^+}(\bR)^{\delta}; \bZ)\to H_2(\mathrm{B}\text{GL}_2^+(\bR)^{\delta}; \bZ),
\]
is split injective with a cokernel which is isomorphic to $\bZ$. 

The map $\bR^+\times \bR^+\cong D\to A\to \text{GL}_2^+(\bR)$ induces the map
\[
\Exterior^2\bR^+\oplus (\bR^+\otimes \bR^+)\oplus \Exterior^2\bR^+\cong H_2(\mathrm{B} \bR^{+,\delta}\times \mathrm{B} \bR^{+,\delta}; \bZ)\to H_2(\mathrm{B} \text{GL}_2^+(\bR)^{\delta}; \bZ).
\]
Let $\sigma$ be the involution of diagonal entries of $D\cong \bR^+\times \bR^+$. The spectral sequence in the proof of \cite[Chapter 6, Theorem 5.7]{weibel2013k} implies that this map factors through the co-invariants
\[
H_2(\mathrm{B} \bR^{+,\delta}\times \mathrm{B} \bR^{+,\delta}; \bZ)_{\sigma}.
\]
The group $H_2(\mathrm{B} \bR^{+,\delta}\times \mathrm{B} \bR^{+,\delta}; \bZ)_{\sigma}$ is isomorphic to $\Exterior^2\bR^+\oplus \tilde{\Exterior}^2\bR^+$ where $\tilde{\Exterior}^2\bR^+$ denotes the quotient of the group $\bR^+\otimes \bR^+$ by the
subgroup generated by all $a\otimes b+b\otimes a$. The proof of \cite[Chapter 6, Theorem 5.7]{weibel2013k} also implies that the restriction of the map
\[
\Exterior^2\bR^+\oplus \tilde{\Exterior}^2\bR^+\to H_2(\mathrm{B} \text{GL}_2^+(\bR)^{\delta}; \bZ),
\]
 on the summand $\Exterior^2\bR^+$ is injective and maps $\tilde{\Exterior}^2\bR^+$ surjectively to the summand $K_2(\bR)$ in $H_2(\mathrm{B} \text{GL}_2^+(\bR)^{\delta}; \bZ)$. So we summarize what we need from Suslin's calculation in the following lemma.
\begin{lem}\label{Suslin}
The map
 \[
  \bR^+\otimes \bR^+\to H_2(\mathrm{B} \bR^{+,\delta}\times \mathrm{B} \bR^{+,\delta}; \bZ)\to H_2(\mathrm{B} \text{GL}_2^+(\bR)^{\delta}; \bZ),
  \]
  surjects to the image of $K_2(\bR)\cong H_2(\mathrm{B}\widetilde{\text{SL}_2}(\bR)^{\delta}; \bZ)\to H_2(\mathrm{B} \text{GL}_2^+(\bR)^{\delta}; \bZ)$.
\end{lem}

 
\begin{proof}[Proof of \Cref{X}] Since $X$ is simply connected, to prove that it is $2$-connected, we need to show that $H_2(X;\bZ)=0$.  The homotopy pushout in diagram \ref{2} gives the Mayer-Vietoris sequence
\begin{equation}\label{eq:8}
 \begin{tikzcd}
  H_2(\mathrm{B}A^{\delta}; \bZ) \rar{i_2} & H_2(R_A; \bZ)\oplus H_2(\mathrm{B}\widetilde{\text{GL}_2^+}(\bR)^{\delta}; \bZ) \rar
             \ar[draw=none]{d}[name=X, anchor=center]{}
    & H_2(X; \bZ) \ar[rounded corners,
            to path={ -- ([xshift=2ex]\tikztostart.east)
                      |- (X.center) \tikztonodes
                      -| ([xshift=-2ex]\tikztotarget.west)
                      -- (\tikztotarget)}]{dll}[at end]{\delta} \\      
  H_1(\mathrm{B}A^{\delta}; \bZ) \rar{i_1} & H_1(R_A; \bZ)\oplus H_1(\mathrm{B}\widetilde{\text{GL}_2^+}(\bR)^{\delta}; \bZ) \rar & H_1(X; \bZ)=0
\end{tikzcd}
\end{equation}
First, we observe that $i_1$ is an isomorphism. From \Cref{iota_1}, we know that the kernel of the map
\[
\alpha_1\colon H_1(\mathrm{B}A^{\delta}; \bZ)\to H_1(R_A; \bZ),
\]
is given by the image of $H_1(\mathrm{B}\text{Aff}^+(\bR)^{\delta}; \bZ)\to H_1(\mathrm{B}A^{\delta}; \bZ)$. So to prove that $i_1$ is an isomorphism, it is enough to show the composition 
\[
\mathrm{B}\text{Aff}^+(\bR)^{\delta}\to \mathrm{B}A^{\delta}\to \mathrm{B}\widetilde{\text{GL}_2^+}(\bR)^{\delta},
\]
induces an isomorphism on the first homology. On the other hand, using the isomorphism $f\colon \text{SL}_2(\bR)\times \bR^{+}\xrightarrow{\cong} \text{GL}_2^+(\bR)$, we know that 
\[
H_1(\mathrm{B}\text{Aff}^+(\bR)^{\delta}; \bZ)\to H_1(\mathrm{B}\text{GL}_2^+(\bR)^{\delta}; \bZ),
\]
is an isomorphism. Hence, to prove that $i_1$ is an isomorphism, it is enough to show that 
\[
H_1(\mathrm{B}\widetilde{\text{GL}_2^+}(\bR)^{\delta}; \bZ)\to H_1(\mathrm{B}\text{GL}_2^+(\bR)^{\delta}; \bZ),
\]
is an isomorphism. The Serre spectral sequence for the fibration
\begin{equation}\label{ss}
\bS^1\to \mathrm{B}\widetilde{\text{GL}_2^+}(\bR)^{\delta}\to \mathrm{B}\text{GL}_2^+(\bR)^{\delta},
\end{equation}
gives the long exact sequence
\begin{equation}\label{eq5}
 \begin{tikzcd}
  H_2(\mathrm{B}\widetilde{\text{GL}_2^+}(\bR)^{\delta}; \bZ) \rar{u} & H_2(\mathrm{B}\text{GL}_2^+(\bR)^{\delta}; \bZ) \rar{e}
             \ar[draw=none]{d}[name=X, anchor=center]{}
    & \bZ \ar[rounded corners,
            to path={ -- ([xshift=2ex]\tikztostart.east)
                      |- (X.center) \tikztonodes
                      -| ([xshift=-2ex]\tikztotarget.west)
                      -- (\tikztotarget)}]{dll}[at end]{} \\      
  H_1(\mathrm{B}\widetilde{\text{GL}_2^+}(\bR)^{\delta}; \bZ) \rar{v} & H_1(\mathrm{B}\text{GL}_2^+(\bR)^{\delta}; \bZ) \rar & 0.
\end{tikzcd}
\end{equation}
The map $e$ is the Euler class for flat $\text{GL}_2^+(\bR)$-bundles over surfaces. By Minlor's theorem (\cite[Theorem 2]{MR0095518}), the Euler number of flat $\text{GL}_2^+(\bR)$-bundles over a surface of genus $g$ can take any value between $-g+1$ and $g-1$. So by varying $g$, we conclude that $e$ is surjective (it is in fact split surjective). Therefore, the map $v$ is an isomorphism.


 So to prove that $H_2(X;\bZ)=0$, it is enough to show that 
\begin{equation}\label{eq6}
i_2\colon H_2(\mathrm{B}A^{\delta}; \bZ)\to H_2(R_A; \bZ)\oplus H_2(\mathrm{B}\widetilde{\text{GL}_2^+}(\bR)^{\delta}; \bZ),
\end{equation}
is a surjection.  In \Cref{iota_1}, we determined the kernel of the split surjective map
\[
\alpha_2\colon H_2(\mathrm{B}A^{\delta}; \bZ)\to H_2(R_A; \bZ).
\]
Hence, it is enough to show that $\ker(\alpha_2)\to H_2(\mathrm{B}\widetilde{\text{GL}_2^+}(\bR)^{\delta}; \bZ)$ is surjective. 

%
Recall that 
\[
u\colon H_2(\mathrm{B}\widetilde{\text{GL}_2^+}(\bR)^{\delta}; \bZ)\to H_2(\mathrm{B}\text{GL}_2^+(\bR)^{\delta}; \bZ),
\]
is split injective where $\text{Im}(u)\cong K_2(\bR)\oplus  \Exterior^2 \bR^{+}$ and by the above discussion the cokernel is isomorphic to $\bZ$ via the map $e$ in the long exact sequence \ref{eq5}. So we need to show that $\ker(\alpha_2)\cong (\bR^+\otimes\bR^+)\oplus \Exterior^2\bR^+$ maps surjectivey to the summand $\text{Im}(p)$.

Recall from \Cref{iota_1} that the map $t\colon \bR^+\to A$ given by $t(a)=\begin{bmatrix}
1& 0 \\
0 & a 
\end{bmatrix}$ induces a map on the second homology groups
\[
 \Exterior^2 \bR^{+}\to \ker(\alpha_2),
\]
which is isomorphic to the summand $ \Exterior^2 \bR^{+}$ in $\ker(\alpha_2)$ in \Cref{iota_1}. On other hand, under the isomorphism $f\colon \text{SL}_2(\bR)\times \bR^{+}\xrightarrow{\cong} \text{GL}_2^+(\bR)$, the matrix $t(a)$ comes from $(\text{Id}, \sqrt{a})$. Given that the square root is an isomorphism of $\bR^+$, we obtain that the composition of maps
\[
H_2(\mathrm{B}\bR^{+, \delta}; \bZ)\xrightarrow{t_*} H_2(\mathrm{B}A^{\delta}; \bZ)\to H_2(\mathrm{B}\text{GL}_2^+(\bR)^{\delta}; \bZ)\cong H_2(\mathrm{B}\text{SL}_2(\bR)^{\delta}; \bZ)\oplus  \Exterior^2 \bR^{+},
\]
is injective and is isomorphic to the summand $ \Exterior^2 \bR^{+}$. Hence, to finish the proof of surjectivity of 
\[
 (\bR^+\otimes\bR^+)\oplus \Exterior^2\bR^+\cong \ker(\alpha_2)\to \text{Im}(p)\cong K_2(\bR)\oplus  \Exterior^2\bR^+,
\]
 it is enough to prove that the summand $\bR^+\otimes \bR^+$ in $\ker(\alpha_2)$  maps surjectively to $K_2(\bR)$. Recall that the summand  $\bR^+\otimes \bR^+$ in \Cref{iota_1} is induced by embedding of diagonal matrices and using the Kenneth formula
 \[
  \bR^+\otimes \bR^+\to H_2(\mathrm{B} \bR^{+,\delta}\times \mathrm{B} \bR^{+,\delta}; \bZ)\to H_2(\mathrm{B}A^{\delta}; \bZ).
 \]
So from \Cref{Suslin}, it follows that the summand $\bR^+\otimes \bR^+$ in $\ker(\alpha_2)$  maps surjectively to $K_2(\bR)$.
\end{proof}

Recall that we have  natural maps  $\nu\colon \mathrm{B}\Gamma_2^{PL}\to \mathrm{B}\text{PL}^+(\bR^2)$ and $\zeta\colon \mathrm{B}\text{PL}^+(\bR^2)\xrightarrow{\simeq}\mathrm{B}\text{Homeo}^+(\bR^2)\simeq \mathrm{B}S^1$. So they induce a map
$$\psi\colon rX\to \mathrm{B}\text{Homeo}^+(\bR^2).$$
 We think of the map $\psi$ as the map that classifies the normal bundle to codimension $2$ PL Haefliger structures as $\bR^2$-bundles. Therefore, to prove \Cref{PL} which says that $\overline{\mathrm{B}\Gamma}_2^{\textnormal{\text{PL}}}$ is $4$-connected, it is enough to prove the following.

\begin{thm}\label{integer}
The map $\psi$ induces an isomorphism on $H_*(-;\bZ)$ for $*\leq 4$.
\end{thm}
We need another preliminary lemma. In Greenberg's homotopy push-out diagram \ref{3}, there is a map $q\colon LX\hcoker S^1\to rX$ and also there is a natural map $p\colon LX\hcoker S^1\to \mathrm{B}S^1$ that classifies the universal circle bundle over the homotopy quotient $LX\hcoker S^1$. Since $S^1\hookrightarrow \text{Homeo}_0(S^1)$ is homotopy equivalence, we shall consider the following equivalent models for these maps
\[
q\colon LX\hcoker  \text{Homeo}_0(S^1)\to rX\simeq  \mathrm{B}\Gamma_2^{\textnormal{\text{PL}}},
\]
\[
p\colon LX\hcoker  \text{Homeo}_0(S^1)\to \mathrm{B} \text{Homeo}_0(S^1).
\]
There is also the composition $\text{Homeo}_0(S^1)\to \text{Homeo}_0(D^2)\to \text{Homeo}^+(\bR^2)$ where the first map is the Alexander cone construction and the second map is the restriction to the identity. This inclusion  induces a weak homotopy equivalence $$\iota\colon\mathrm{B}\text{Homeo}_0(S^1)\xrightarrow{\simeq} \mathrm{B}\text{Homeo}^+(\bR^2).$$
\begin{lem}\label{compatible}
The maps $\iota\circ p$ and $\psi\circ q$ induce the same map on homology.
\end{lem}
\begin{proof} This is already implicit in Greenberg's paper (\cite{MR1200422}) but for the convenience of the reader, we shall first recall the relevant object for this proof. As in Greenberg's paper (\cite[Page 188]{MR1200422}), let $P_0$ be the group of germs of orientation preserving PL homeomorphisms of $\bR^2$ that fix the origin. Ghys--Sergiescu (\cite[Section 2]{ghys1987groupe} and \cite[Theorem 2.25 and Corollary 2.26]{MR1200422}) proved a general version of the Mather-Thurston homology isomorphism theorem for certain groupoids on the circle. As a result, there is a map
\[
f\colon \mathrm{B}P_0\to LX\hcoker \text{Homeo}_0(S^1),
\]
induces a homology isomorphism. To prove the lemma, we use Greenberg's description of Ghys--Sergiescu's theorem to show that the two maps
\[
\iota\circ p\circ f\colon \mathrm{B}P_0\to \mathrm{B}\text{Homeo}^+(\bR^2),
\]
\[
\psi\circ q\circ f\colon \mathrm{B}P_0\to \mathrm{B}\text{Homeo}^+(\bR^2),
\]
induce isomorphic $\bR^2$-bundles over $\mathrm{B}P_0$. 

By Greenberg's description (\cite[Section 2.22]{MR1200422}), the composition $q\circ f\colon  \mathrm{B}P_0\to \mathrm{B}\Gamma_2^{\textnormal{\text{PL}}}$ is induced by the inclusion of $P_0$ as the group of germs into the groupoid $\Gamma_2^{\textnormal{\text{PL}}}$. One can canonically extend each germ in $P_0$ to a PL homeomorphism of $\bR^2$. So there is a natural action of $P_0$ on $\bR^2$. Therefore, the map $\psi\circ q\circ f$ 
\[
\mathrm{B}P_0\xrightarrow{} \mathrm{B}\Gamma_2^{\textnormal{\text{PL}}}\xrightarrow{\nu} \mathrm{B}\text{PL}^+(\bR^2)\xrightarrow{}\mathrm{B}\text{Homeo}^+(\bR^2),
\]
classifies the $\bR^2$-bundle on $\mathrm{B}P_0$ induced by the action of $P_0$ on $\bR^2$. 

On the other hand, $P_0$ acts on rays out of the origin. So $P_0$ also maps into $\text{PP}(S^1)$ the group of orientation-preserving piecewise projective homeomorphisms of $S^1$. In particular, it is a subgroup of orientation-preserving homeomorphisms of the circle. The map $p\circ f$
\[
\mathrm{B}P_0\to \mathrm{B}\text{Homeo}_0(S^1),
\]
classifies the natural circle bundle over $\mathrm{B}P_0$ induced by the action of $P_0$ on $S^1$. Therefore, the map $\iota\circ p\circ f$ classifies the Euclidean $\bR^2$-bundle induced by the natural action of $P_0$ on $\bR^2$.
\end{proof}

Let $\text{ev}\colon LX\to X$ be the map induced by evaluating loops at the base point $1$ of the unit circle in the complex plane. The circle action $\eta\colon S^1\times LX\to LX$ sends the pair $(s, \gamma(t))$ where $\gamma(t)$ is a free loop in $X$ to the loop $\gamma(st)$. The map $\eta$ induces the map
\[
\Delta_*\colon H_*(LX; \bZ)\to H_{*+1}(LX; \bZ).
\]
For each positive integer $k$, let $h_k\colon \pi_k(\Omega X)\to H_{k+1}(X; \bZ)$ be the map that sends the homotopy class of $f\colon S^k\to \Omega X$ to $F_*([S^1\times S^k])$  where $F\colon S^1\times S^k\to X$ is the map induced by the adjoint of $f$ (it is the adjoint of $f$ composed with swapping $S^1$ and $S^k$ factors). 
\begin{lem}\label{comm}We have a commutative diagram

\[
 \begin{tikzpicture}[node distance=1.5cm, auto]
  \node (A) {$ \pi_*(\Omega X)$};
  \node (B) [right of=A, node distance=6cm] {$H_{*+1}(X; \bZ)$};
  \node (C) [below of=A]{$H_*(\Omega X; \bZ)$};
  \node (D) [right of=C, node distance=3cm]{$H_*(LX; \bZ)$};
    \node (E) [below of=B]{$H_{*+1}(LX; \bZ),$};

  \draw [->] (A) to node {$h_*$}(B);
    \draw [->] (A) to node {$\text{H}$}(C);
  \draw [->] (C) to node {$G$}(D);
  \draw [->] (E) to node {$\text{ev}_*$}(B);
    \draw [->] (D) to node {$\Delta_*$}(E);
\end{tikzpicture}
\]
where the map $H$ is the Hurewicz map and the map $G$ is induced by the inclusion $\Omega X\to LX$. 
\end{lem}
\begin{proof}
Let $f\colon S^k\to \Omega X$ be an element in $\pi_k(\Omega X)$ and let $\tilde{f}$ be $G\circ H(f)\in H_k(LX;\bZ)$. Then $\Delta_k(\tilde{f})$ is defined to be the map
\[
\Delta_k(\tilde{f})\colon S^1\times S^k\to LX,
\]
which sends the pair $(s, x)$ to to the action of $s$ on the loop $f(x)(t)$ which is $f(x)(st)$. The evaluation map evaluates this loop at $t=1$ which gives the same map as the adjoint $F\colon S^1\times S^k\to X$. Hence, we have $\text{ev}_k\circ \Delta_k(\tilde{f})=h_k(f)$.
\end{proof}
\begin{cor}\label{Delta}
Let $X$ be $2$-connected space and \begin{itemize}
\item The map $\Delta_2$ is injective and $\Delta_2(H_2(LX; \bZ))$ maps isomorphically to $H_3(X; \bZ)$ via the evaluation map $\text{ev}$.
\item $\Delta_3(H_3(LX; \bZ))$ maps surjectively to $H_4(X; \bZ)$ via the evaluation map $\text{ev}$.
\end{itemize}

\end{cor}
\begin{proof}
Since $X$ is $2$-connected, the Hurewicz map $\pi_3(X)\to H_3(X;\bZ)$ is an isomorphism and also $\pi_4(X)\to H_4(X; \bZ)$ is surjective. On the other hand, $LX$ is also simply connected, therefore we have the isomorphisms
\[
\pi_2(\Omega X)\xrightarrow{\cong} H_2(\Omega X; \bZ)\xrightarrow{\cong} H_2(LX; \bZ),
\]
where $\Omega X$ is the based loop space on $X$. Also, note the map 
\[
h_2\colon \pi_2(\Omega X)\to H_3(X; \bZ),
\]
is an isomorphism.
From \Cref{comm}, we know that $\text{ev}_2\circ \Delta_2\colon\pi_2(\Omega X)\to H_3(X; \bZ)$ is the same map as $h_2$ which proves the first statement.  

Since $\Omega X$ is simply connected, the Hurewicz map
\[
\pi_3(\Omega X)\to H_3(\Omega X; \bZ),
\]
is surjective. So to prove the second statement, it is enough to show that the composition 
\begin{equation}\label{sur}
\pi_3(\Omega X)\xrightarrow{\cong} H_3(\Omega X; \bZ)\to H_3(LX; \bZ)\xrightarrow{\Delta_3}H_4(LX; \bZ)\xrightarrow{\text{ev}} H_4(X; \bZ),
\end{equation}
is surjective. But again by \Cref{comm} the above composition is the same as the natural map
\[
h_3\colon\pi_3(\Omega X)\to H_4(X; \bZ),
\]
that sends the homotopy class of $f\colon S^3\to \Omega X$ to $F_*([S^1\times S^3])$ where $F\colon S^1\times S^3\to X$ is the map induced by the adjoint of $f$. Now since $X$ is $2$-connected, the map $h_3$ is surjective. Therefore, the composition \ref{sur} is also surjective. 
\end{proof}
\begin{proof}[Proof of \Cref{integer}]
Recall from the introduction that the space $\overline{\mathrm{B}\Gamma}_2^{\textnormal{\text{PL}}}$  which is weakly equivalent to the homotopy fiber of the map
$$\psi\colon rX\to \mathrm{B}\text{Homeo}^+(\bR^2),$$
is known to be at least $2$-connected. So the map $\psi$ induces isomorphisms on $H_*(-;\bZ)$ for $*\leq 2$. Hence, we need to show two things; one is that $H_3(rX; \bZ)=0$ and the other is
$$\psi_4\colon H_4(rX; \bZ)\to H_4(\mathrm{B}\text{Homeo}^+(\bR^2); \bZ)\cong \bZ,$$
is an isomorphism.  First, note that  the Mayer-Vietoris sequence for the pushout \ref{3} gives
\[
H_i(LX)\to H_i(X)\oplus H_i(LX\hcoker S^1)\to H_i(rX)\to H_{i-1}(LX)\to H_{i-1}(X)\oplus H_{i-1}(LX\hcoker S^1).
\]
To compute $H_*(rX;\bZ)$ for $*\leq 4$, we use that $X$ is $2$-connected and the fact that fibrations
\[
\Omega X\to LX\xrightarrow{\text{ev}} X,
\]
\begin{equation}\label{LX}
LX\to LX\hcoker S^1\xrightarrow{p} \mathrm{B}S^1,
\end{equation}
have sections. The first fibration has a section by considering constant loops and the second fibration has a section because the action of $S^1$ has fixed points i.e. the constant loops. Therefore, $H_*(LX)\xrightarrow{\text{ev}_*}H_*(X)$ is surjective and so is $H_*(LX\hcoker S^1)\to H_*(\mathrm{B}S^1)$ and since they have  sections, $H_*(X)$ and $H_*(\mathrm{B}S^1)$ split off as summands from $H_*(LX)$ and $H_*(LX\hcoker S^1)$ respectively. 


\begin{figure}[h]
\[
\begin{tikzpicture}
\draw [<-] (-0.6,4.7) -- (-0.6,0.2);
\draw [->] (-1,0.6)-- (9.6,0.6);
\node  at (0.2,0.9) {\small$\bZ$};
\node  at (0.2,0.3) {$0$};
\node at (2.2,0.9) {\small$0$};
\node at (2.2,0.3) {\small$1$};


\node at (4.2,0.9) {\small$\bZ$};
\node at (4.2,0.3) {\small$2$};
\node at (6.2,0.3) {\small$3$};
\node at (8.2,0.3) {\small$4$};

\node at (6.2,0.9) {\small$0$};
\node at (8.2,0.9) {\small$\bZ$};

\node at (6.2,1.5) {\small$0$};
\node at (8.2,1.5) {\small$0$};



\node at (9.8,0.3) {\small$p$};

\node at (0.2,1.5) {\small$0$};
\node at (-1,1.5) {\small$1$};
\node at (-1,0.9) {\small$0$};

\node at (0.2,2.2) {\small$H_2(LX)$};
\node at (-1,2.2) {\small$2$};



\node at (0.2,5.2) {\small$$};
\node at (-1,2.9) {\small$3$};
\node at (-1, 3.6) {\small$4$};
\node at (0.2, 3.6) {\small$H_4(LX)$};
\node at (0.2,2.9) {\small$H_3(LX)$};
\node at (-1,5) {\small$q$};

\node at (2.2,2.9) {\small$0$};

\node at (2.2,2.2) {\small$0$};
\node at (4.2,2.2) {\small$H_2(LX)$};

\node at (2.2,1.5) {\small$0$};
\node at (4.2,1.5) {\small$0$};


\draw [<-] (0.9,2.9)--(3.4,2.3);
\draw [<-] (0.9,2.2)--(3.4,1.6);
\draw [<-] (0.9,1.5)--(3.4,0.9);

\draw [<-] (4.9,2.2)--(7.4,1.6);
\draw [<-] (4.9,1.5)--(7.4,0.9);

\end{tikzpicture}\]
\caption{ The second page of the Serre spectral sequence for the fibration $LX\to LX\hcoker S^1\to \mathrm{B}S^1$.}\label{s''}
\end{figure}
From the Serre spectral sequence for the fibration \ref{LX}, we see that $H_2(LX; \bZ)\to H_2(LX\hcoker S^1; \bZ)$ is injective. So to show that $H_3(rX; \bZ)=0$, it is enough to prove that 
\begin{equation}\label{eq1}
H_3(LX)\to H_3(X)\oplus H_3(LX\hcoker S^1),
\end{equation}
is surjective.

Note that the differentials out of  $\bZ$'s in the $0$-th row are trivial because of the existence of the section for the map $p$ in the fibration \ref{LX}. And it is standard that the differentials
\[
d_2\colon H_i(LX)\to H_{i+1}(LX),
\]
are the same as the map $\Delta_i$ in \Cref{Delta} (\cite[Proposition 3.3]{MR1736363}). 

%
From the first part of \Cref{Delta}, we know that the map $d_2$ in 
\[
H_2(LX; \bZ)\xrightarrow{d_2} H_3(LX; \bZ)\xrightarrow{\text{ev}_*}H_3(X; \bZ).
\]
is injective and the natural map $\text{ev}_*\colon d_2(H_2(LX; \bZ))\to H_3(X; \bZ)$ is an isomorphism. Given that $d_2(H_2(LX; \bZ))$ is the kernel of the surjection $H_3(LX)\twoheadrightarrow H_3(LX\hcoker S^1)$, the map \ref{eq1} is in fact an isomorphism. So we have $H_3(rX; \bZ)=0$. 

Now since the map \ref{eq1} is an isomorphism, to show that $\psi$ induces an isomorphism on $H_4(-;\bZ)$, it is enough to show that the cokernel of the map
\begin{equation}\label{4}
H_4(LX; \bZ)\to H_4(X; \bZ)\oplus H_4(LX\hcoker S^1; \bZ),
\end{equation}
is the $\bZ$ summand in $H_4(LX\hcoker S^1; \bZ)$ coming from $0$-row in the Serre spectral sequence. This is because, in that case, the composition 
\[
\mathrm{B}S^1\to LX\hcoker S^1\to rX,
\]
where the first map is the section of $p$, induces an isomorphism on $H_4(-;\bZ)$; and \Cref{compatible} implies that the composition 
\[
\mathrm{B}S^1\to LX\hcoker S^1\to rX\xrightarrow{\psi}\mathrm{B}\text{Homeo}^+(\bR^2),
\]
induces a homology isomorphism.

To do this, from the second part of \Cref{Delta}, we know that  in 
\[
H_3(LX; \bZ)\xrightarrow{d_2} H_4(LX; \bZ)\xrightarrow{\text{ev}_*}H_4(X; \bZ).
\]
 $d_2(H_3(LX; \bZ))$ surjects to $H_4(X; \bZ)$ via $\text{ev}_*$. Since $H_4(LX\hcoker S^1; \bZ)$ is isomorphic to $\bZ\oplus H_4(LX; \bZ)/d_2(H_3(LX; \bZ))$, the cokernel of the map 
\[
H_4(LX; \bZ)\to H_4(X; \bZ)\oplus H_4(LX\hcoker S^1; \bZ),
\]
is the $\bZ$ summand in $H_4(LX\hcoker S^1; \bZ)$. Hence, $\psi$ induces an isomorphism on $H_4(-;\bZ)$.
\end{proof}
\begin{quest}
Is there a ``discrete" Godbillon-Vey class similar to the case codimenison $1$ PL-foliations in \cite{ghys1987groupe}, to give a nontrivial map $H_5(\overline{\mathrm{B}\Gamma}_2^{\textnormal{\text{PL}}};\bZ)\to \bR$?
\end{quest}
\subsection{Homology of PL surface homeomorphisms made discrete.} To relate the group homology of PL surface homeomorphisms to the homotopy type of $\overline{\mathrm{B}\Gamma}_2^{\textnormal{\text{PL}}}$, we first recall a version of Mather-Thurston's theorem that the author proved (\cite[Section 5]{nariman2020thurston}).  Let $M$ be an $n$-dimensional PL-manifold possibly with a non-empty boundary. The topological group $\text{PL}(M, \text{rel }\partial)$ is the realization of the simplicial group $S_{\bullet}\text{PL}(M, \text{rel }\partial)$ whose $k$-simplices are given by the set of PL homeomorphisms of $\Delta^k\times M$ that commute with the projection to the first factor and whose supports are away from the boundary of $M$. We have the map $\text{PL}(M, \text{rel }\partial)^{\delta}\to \text{PL}(M, \text{rel }\partial)$ given by the inclusion of $0$-simplices. This map induces the map between classifying spaces
\[
\mathrm{B}\text{PL}(M, \text{rel }\partial)^{\delta}\to \mathrm{B}\text{PL}(M, \text{rel }\partial),
\]
whose homotopy fiber is denoted by $\overline{\mathrm{B}\text{PL}(M, \text{rel }\partial)}$. This homotopy fiber can also be described as the realization of the semi-simplicial set $S_{\bullet}(\overline{\mathrm{B}\text{PL}(M, \text{rel }\partial)})$ whose $k$-simplices are given by the set of codimension $n$ foliations on $\Delta^k \times M $ that are transverse to the fibers of the projection $ \Delta^k\times M\to \Delta^k$ and the holonomies are compactly supported PL-homeomorphisms of the fiber $M$. Note that $S_{\bullet}\text{PL}(M, \text{rel }\partial)$ acts level-wise on the simplices $S_{\bullet}(\overline{\mathrm{B}\text{PL}(M, \text{rel }\partial)})$. Hence, we have an action of $\text{PL}(M, \text{rel }\partial)$ on $\overline{\mathrm{B}\text{PL}(M, \text{rel }\partial)}$.

  On the other hand,  the space $\overline{\mathrm{B}\text{PL}(M, \text{rel }\partial)}$ is related to the classifying space of the groupoid $\Gamma_n^{\text{PL}}$ as follows. Recall that forgetting the germ of the foliation of foliated microbundles induces the map
\[
\nu\colon\mathrm{B}\Gamma_n^{\text{PL}}\to \mathrm{B}\text{PL}(\bR^n),
\]
between classifying spaces. Let $\tau_M\colon M\to \mathrm{B}\text{PL}(\bR^n)$ be a map that classifies the tangent microbundle of $M$. Let $\text{Sect}_{\partial}(\tau_M^*(\nu))$ be the space of sections of the pullback bundle $\tau_M^*(\nu)$ over $M$ that are supported away from the boundary. The support of a section is measured with respect to a fixed base section. By the obstruction theory and the fact that the fiber of the bundle $\tau_M^*(\nu)$ over $M$ is at least $n$-connected, the space of sections is connected. So different choices of a base section do not change the homotopy type of the compactly supported sections. 

We recall from (\cite[Section 1.2.2]{nariman2015stable}) how $\text{PL}(M, \text{rel }\partial)$ acts on $\text{Sect}_{\partial}(\tau_M^*(\nu))$. The PL tangent microbundle of the PL manifold $M$ is the following microbundle
\[
M\xrightarrow{\Delta}M\times M\xrightarrow{pr_1}M.
\]
A germ of PL foliation $c$ on $\Delta^p\times M\times M$ at $\Delta^p\times \text{diag }M$ which is transverse to the fiber of the projection $id\times pr_1: \Delta^p\times M\times M\to \Delta^p\times M$, is said to be horizontal at $x\in M$ if there exists a neighborhood $U$ around $x$ such that the restriction of the foliation $c$ to $\Delta^p\times U\times U$ is induced by the projection $\Delta^p\times U\times U\to U$ on the last factor. By the support of $c$, we mean the set of $x\in M$, where $c$ is not horizontal. 

Now we define the semi-simplicial set $S_{\bullet}(  \mathrm{Sect}_c(\tau_M^*(\nu)))$  whose $p$-simplices are given as the set of germs of PL foliations on $\Delta^p\times M\times M$ at $\Delta^p\times \text{diag }M$ which are transverse to the fiber of the projection $id\times pr_1: \Delta^p\times M\times M\to \Delta^p\times M$ and have compact support. The realization of this semi-simplicial set gives a model for the compactly supported sections $ \mathrm{Sect}_c(\tau_M^*(\nu))$. Similar to the previous case, there is an obvious action of $S_{\bullet}(\text{PL}(M, \text{rel }\partial))$ on $S_{\bullet}(  \mathrm{Sect}_c(\tau_M^*(\nu)))$.

In this model, there is a natural map (\cite[Section 1.2.2]{nariman2015stable})
\begin{equation}\label{MT1}
\overline{\mathrm{B}\text{PL}(M, \text{rel }\partial)}\to \text{Sect}_{\partial}(\tau_M^*(\nu)),
\end{equation}
that is $\textnormal{\text{PL}}(M, \text{rel }\partial)$-equivariant and we showed that it induces a homology isomorphism (\cite{nariman2020thurston}). \footnote{Gael Meigniez also told the author that his method in \cite{meigniez2021quasicomplementary} can be used to prove the PL version of Mather-Thurston's theorem.} Therefore, the homotopy quotients of the action of $\textnormal{\text{PL}}(M, \text{rel }\partial)$-equivariant  on both sides also induces a homology isomorphism. Hence, $\mathrm{B}{\textnormal{\text{PL}}}^{\delta}(M,\text{rel }\partial)$ is homology isomorphic to $\text{Sect}_{\partial}(\tau_M^*(\nu))\hcoker  \textnormal{\text{PL}}(M, \text{rel }\partial)$.
\begin{proof}[Proof of \Cref{perfectness}]  Let $\Sigma$ be an oriented  closed surface possibly with non-empty boundary. To show that the map
\[
\mathrm{B}{\textnormal{\text{PL}}}^{\delta}_0(\Sigma,\text{rel }\partial)\to \mathrm{B}{\textnormal{\text{PL}}}_0(\Sigma,\text{rel }\partial),
\]
induces an isomorphism on $H_*(-;\bZ)$ for $*\leq 2$ and a surjection on $H_3(-;\bZ)$, it is enough to show that $H_*(\overline{ \mathrm{B}{\textnormal{\text{PL}}}(\Sigma,\text{rel }\partial)};\bZ)$ vanishes for $*\leq 2$. By the Mather-Thurston's theorem described above, these groups are isomorphic to $H_*(\text{Sect}_{\partial}(\tau_{\Sigma}^*(\nu));\bZ)$. Hence, it is enough to show that $\text{Sect}_{\partial}(\tau_{\Sigma}^*(\nu))$ is $2$-connected. Note that in general, if the fiber of a fibration $\pi\colon E\to M^n$ is $k$ connected, elementary obstruction theory argument implies that the space of section of $\pi$ is $(k-n)$-connected. Now recall that  the homotopy fiber of the fibration $\tau_{\Sigma}^*(\nu)\to \Sigma$ is $\overline{\mathrm{B}\Gamma}_2^{\textnormal{\text{PL}}}$ which is $4$-connected by \Cref{PL}. Therefore,  the space $\text{Sect}_{\partial}(\tau_{\Sigma}^*(\nu))$ is $2$-connected. 
\end{proof}
\begin{rem}
Calegari and Rolfsen proved in particular the local indicability of PL homeomorphisms (\cite[Theorem 3.3.1]{calegari2015groups}) of manifolds relative to the non-empty boundary. As a consequence of their local indicability result, one deduces that no finitely generated subgroup of ${\textnormal{\text{PL}}}^{\delta}_{c,0}(\Sigma)$ is a simple group. However, our theorem shows that the ambient group ${\textnormal{\text{PL}}}^{\delta}_{c,0}(\Sigma)$ is simple which is similar to Thurston's stability for  $C^1$-diffeomorphisms groups (\cite{thurston1974generalization}). 
\end{rem}
\begin{rem}
We know that ${\textnormal{\text{PL}}}_0(\Sigma,\text{rel }\partial)\simeq \Diff_0(\Sigma,\text{rel }\partial)$ (\cite[Page 8]{MR652596}). Given that the homotopy type of $\Diff_0(\Sigma,\text{rel }\partial)$ is completely known for all surfaces (\cite{earle1969fibre}), we could also compute the second group homology of ${\textnormal{\text{PL}}}^{\delta}_0(\Sigma,\text{rel }\partial)$ with $\bF_p$ coefficients. For example if $\Sigma$ is a hyperbolic surface, we obtain $H_2(\mathrm{B}{\textnormal{\text{PL}}}^{\delta}_0(\Sigma,\text{rel }\partial);\bZ)$ is trivial. 
\end{rem}
In this dimension, it is known (\cite[Page 8]{MR652596}) that $\text{PL}^+(\bR^2)\simeq \text{SO}(2)$. Therefore, $H^*( \mathrm{B}\text{PL}^+(\bR^2);\bQ)$ is generated by the Euler class $e\in H^2(\mathrm{B}\text{PL}^+(\bR^2);\bQ)$. A consequence of our computation with Greenberg's model is the following non-vanishing result.
\begin{thm}\label{Euler}
The classes $\nu^*(e^k)\in H^{2k}(\mathrm{B}\Gamma_2^{\textnormal{\text{PL}}};\bQ)$ are nontrivial for all $k$.
\end{thm}
This is in contrast with the smooth case. Since in the smooth case we also have the Euler class $\nu^*(e)\in H^2(\mathrm{B}\Gamma_2;\bQ)$ and as a consequence of the Bott vanishing theorem $\nu^*(e^4)$ vanishes in $H^8(\mathrm{B}\Gamma_2;\bQ)$ and it is still open whether $\nu^*(e^3)$ is non-trivial for smooth foliations (\cite[Problem F.5.3]{MR1271828}).
Instead of the identity component, if we consider the entire group ${\textnormal{\text{PL}}}^{\delta}(\Sigma,\text{rel }\partial)$, using \Cref{Euler} and the method in \cite[Theorem 0.4]{nariman2015stable}, we can prove the following non-vanishing result in {\it the stable range}.

\begin{thm}\label{MW}
Let $\Sigma$ be a compact orientable surface, then the map
\[
H^*(\mathrm{B}{\textnormal{\text{PL}}}(\Sigma,\text{rel }\partial);\bQ)\to H^*(\mathrm{B}{\textnormal{\text{PL}}}^{\delta}(\Sigma,\text{rel }\partial);\bQ),
\]
induces an injection when $*\leq (2g(\Sigma)-2)/3$ where $g(\Sigma)$ is the genus of the surface $\Sigma$.
\end{thm} 
As a consequence of Madsen-Weiss' theorem (\cite{madsen2007stable}), $H^*(\mathrm{B}{\textnormal{\text{PL}}}(\Sigma,\text{rel }\partial);\bQ)$ is isomorphic to the polynomial ring $\bQ[\kappa_1,\kappa_2,\dots]$ in the stable range i.e. $*\leq (2g(\Sigma)-2)/3$.  Here $\kappa_i$ are certain characteristic classes of surface bundles known as $i$-th MMM classes whose degree is $2i$. This is also in contrast with the case of smooth diffeomorphisms. Since, in particular, we have the following nonvanishing result.
\begin{cor}
$\kappa_i$ are all non-trivial in $H^*(\mathrm{B}{\textnormal{\text{PL}}}^{\delta}(\Sigma,\text{rel }\partial);\bQ)$ as long as $i\leq (2g(\Sigma)-2)/6$.
\end{cor}
Since the proof of \Cref{MW} uses some background from \cite[Section 1 and 2]{nariman2015stable}, we shall first explain how to adapt these techniques to the case of PL homeomorphisms of surfaces. The proofs in \cite{nariman2015stable} are formulated for diffeomorphism groups of surfaces, but since in dimension $2$, diffeomorphism group of a surface has the same homotopy type as of its PL homeomorphism group, we can use the results of \cite{nariman2015stable} as follows.

In dimension $2$, we know (\cite[Page 8]{MR652596}) that $\phi\colon\mathrm{B}\text{GL}^+_2(\bR)\to \mathrm{B}\text{PL}^+_2(\bR)$ and $ \eta\colon\mathrm{B}\text{PL}^+_2(\bR)\to\mathrm{B}\text{Top}^+_2$ are weak homotopy equivalences where $\text{Top}^+_2$ is the group of orientation preserving homeomorphisms of $\bR^2$. Let $\rho$ be a homotopy inverse to $\eta$. We shall consider the following tangential structures
\[
\nu\colon \mathrm{B}\Gamma_2^{\text{PL}}\to \mathrm{B}\text{PL}^+(\bR^2),
\] 
\[
\nu^s\colon \phi^*(\nu)\to \mathrm{B}\text{GL}^+_2(\bR),
\] 
\[
\nu^t\colon \rho^*(\nu)\to \mathrm{B}\text{Top}^+_2,
\] 
where $\phi^*(\nu)$ and $\rho^*(\nu)$ are the pullbacks of the fibration induced by $\nu$ with the homotopy fiber $\overline{\mathrm{B}\Gamma}_2^{\textnormal{\text{PL}}}$.  Hence, these fibrations are all fiber homotopy equivalent. Let $\tau^s_{\Sigma}, \tau^p_{\Sigma}$ and $\tau^t_{\Sigma}$ be the map classifying the tangent (micro)bundles in the smooth, PL and topological category respectively. 

Given the above homotopy equivalences, the space of sections $\text{Sect}_{\partial}((\tau^s_{\Sigma})^*(\nu^s))$, $\text{Sect}_{\partial}((\tau^p_{\Sigma})^*(\nu))$ and $\text{Sect}_{\partial}((\tau^t_{\Sigma})^*(\nu^t))$ are also all homotopy equivalent. Now since we have $\Diff(\Sigma, \text{rel }\partial)\simeq \text{Homeo}(\Sigma, \text{rel }\partial)\simeq \text{PL}(\Sigma, \text{rel }\partial)$ (\cite[Page 8]{MR652596}), we get a zig-zag of weak homotopy equivalences

   \[
 \begin{tikzpicture}[node distance=4.3cm, auto]
  \node (A) {$|S_{\bullet}(  \mathrm{Sect}_c((\tau^s_{\Sigma})^*(\nu^s)))|\hcoker |S_{\bullet}(\text{Diff}(\Sigma, \text{rel }\partial))|$};
  \node (D) [below of=A, node distance=1.9cm]{$$};
  \node (B) [right of=D,node distance=3cm]{$|S_{\bullet}(  \mathrm{Sect}_c((\tau^t_{\Sigma})^*(\nu^t)))|\hcoker |S_{\bullet}(\text{Homeo}(\Sigma, \text{rel }\partial))|.$};
  \node (C) [right of=A, node distance=7cm]{$|S_{\bullet}(  \mathrm{Sect}_c((\tau^p_{\Sigma})^*(\nu)))|\hcoker |S_{\bullet}(\text{PL}(\Sigma, \text{rel }\partial))|$};
  \draw [<-] (B) to node {$\simeq$} (C);
  \draw [->] (A) to node {$\simeq$} (B);
 \end{tikzpicture}
\]
But as we explained the analog of Mather-Thurston's theorem for PL homeomorphisms, the equivariance of the homology equivalence \ref{MT1} implies that $|S_{\bullet}(  \mathrm{Sect}_c((\tau^p_{\Sigma})^*(\nu)))|\hcoker |S_{\bullet}(\text{PL}(\Sigma, \text{rel }\partial))|$ is homology isomorphic to $\mathrm{B}{\textnormal{\text{PL}}}^{\delta}(\Sigma,\text{rel }\partial)$. Given the above zig-zag,  we have the following lemma.
\begin{lem}
The classifying space $\mathrm{B}{\textnormal{\text{PL}}}^{\delta}(\Sigma,\text{rel }\partial)$ is homology isomorphic to $|S_{\bullet}(  \mathrm{Sect}_c((\tau^s_{\Sigma})^*(\nu^s)))|\hcoker |S_{\bullet}(\textnormal{\text{Diff}}(\Sigma, \text{rel }\partial))|$.
\end{lem}
On the other hand, there is a natural map (\cite[Equation 1.12]{nariman2015stable}) from $|S_{\bullet}(  \mathrm{Sect}_c((\tau^s_{\Sigma})^*(\nu^s)))|\hcoker |S_{\bullet}(\textnormal{\text{Diff}}(\Sigma, \text{rel }\partial))|$ to the moduli space $\mathcal{M}^{\nu^s}(\Sigma)$ of tangential $\nu^s$-structures on $\Sigma$ defined in \cite[Section 1.2.3]{nariman2015stable}. The space $\mathcal{M}^{\nu^s}(\Sigma)$ has been very well studied and as we used the techniques in \cite{randal2009resolutions, galatius2009homotopy, galatius2010monoids}, it exhibits homological stability. As we shall recall in the following proof, its stable homology is described in terms of the Madsen-Tillman spectrum of the corresponding tangential structure. 

\begin{proof}[Proof of \Cref{Euler} and \Cref{MW}]
The key point in \Cref{compatible} is that Greenberg's model for $\mathrm{B}\Gamma_2^{\text{PL}}$ allows up to homotopy, to find a section for the map
\[
\nu\colon \mathrm{B}\Gamma_2^{\text{PL}}\to \mathrm{B}S^1.
\]
Recall that  a null-homotopic map $X\to \mathrm{B}S^1$ and the natural map $LX\hcoker S^1\to \mathrm{B}S^1$ induces a map $rX\to \mathrm{B}S^1$. But a section to the map $LX\hcoker S^1\to \mathrm{B}S^1$ induces a section for $rX\to \mathrm{B}S^1$. Therefore, $\nu^*(e^k)\in H^{2k}( \mathrm{B}\Gamma_2^{\text{PL}};\bQ)$ are nontrivial for all $k$. 

Now let $\gamma$ be the tautological $2$-plane bundle over $\mathrm{B}S^1$ and let $\text{MT}\nu^s$ be the Thom spectrum of the virtual bundle $(\nu^s)^*(-\gamma)$. And let $\Omega^{\infty}_0 \text{MT}\nu^s$ be the base point component of the infinite loop space associated with this Thom spectrum. Then exactly the same method as in \cite[Theorem 0.4]{nariman2015stable} goes through to show that there is a map
\[
\mathrm{B}\text{PL}(\Sigma,\text{rel }\partial)^{\delta}\to \Omega^{\infty}_0 \text{MT}\nu^s,
\]
which is homology isomorphism up to degrees $*\leq (2g(\Sigma)-2)/3$. Hence, we obtain that 
\[
H^*(\mathrm{B}\text{PL}(\Sigma,\text{rel }\partial)^{\delta};\bQ)\cong \text{Sym}^*(H^{*>2}(\mathrm{B}\Gamma_2^{\text{PL}};\bQ)[-2]),
\]
in degrees $*\leq (2g(\Sigma)-2)/3$ where the right hand side is the polynomial ring with the generators in the graded vector space $H^{*>2}(\mathrm{B}\Gamma_2^{\text{PL}};\bQ)$ which is shifted by degree $2$. Since by Madsen-Weiss theorem $H^*(\mathrm{B}\text{PL}(\Sigma,\text{rel }\partial);\bQ)$ is isomorphic to $\text{Sym}^*(H^{*>2}(\mathrm{B}S^1;\bQ)[-2])$ in the same stable range and we know that 
\[
H^{*}(\mathrm{B}S^1;\bQ)\to H^{*}(\mathrm{B}\Gamma_2^{\text{PL}};\bQ),
\]
is injective, then so is 
\[
H^*(\mathrm{B}\text{PL}(\Sigma,\text{rel }\partial);\bQ)\to H^*(\mathrm{B}\text{PL}(\Sigma,\text{rel }\partial)^{\delta};\bQ),
\]
in stable range.
\end{proof}

\bibliographystyle{alpha}
\bibliography{reference}

\begin{thebibliography}{GMTW09}

\bibitem[BH81]{MR652596}
W.~Balcerak and B.~Hajduk.
\newblock Homotopy type of automorphism groups of manifolds.
\newblock {\em Colloq. Math.}, 45(1):1--33 (1982), 1981.

\bibitem[BO99]{MR1736363}
Marcel B\"{o}kstedt and Iver Ottosen.
\newblock Homotopy orbits of free loop spaces.
\newblock {\em Fund. Math.}, 162(3):251--275, 1999.

\bibitem[Bot72]{bott1972lectures}
Raoul Bott.
\newblock {\em Lectures on characteristic classes and foliations}.
\newblock Springer, 1972.

\bibitem[CR15]{calegari2015groups}
Danny Calegari and Dale Rolfsen.
\newblock Groups of pl homeomorphisms of cubes.
\newblock In {\em Annales de la Facult{\'e} des sciences de Toulouse:
  Math{\'e}matiques}, volume~24, pages 1261--1292, 2015.

\bibitem[dLHM83]{de1983acyclic}
Pierre de~La~Harpe and Dusa McDuff.
\newblock Acyclic groups of automorphisms.
\newblock {\em Commentarii Mathematici Helvetici}, 58(1):48--71, 1983.

\bibitem[EE69]{earle1969fibre}
Clifford~J Earle and James Eells.
\newblock A fibre bundle description of teichm{\"u}ller theory.
\newblock {\em Journal of Differential Geometry}, 3(1-2):19--43, 1969.

\bibitem[Eps70]{MR267589}
D.~B.~A. Epstein.
\newblock The simplicity of certain groups of homeomorphisms.
\newblock {\em Compositio Math.}, 22:165--173, 1970.

\bibitem[GMTW09]{galatius2009homotopy}
S{\o}ren Galatius, Ib~Madsen, Ulrike Tillmann, and Michael Weiss.
\newblock The homotopy type of the cobordism category.
\newblock {\em Acta mathematica}, 202(2):195--239, 2009.

\bibitem[Gre92]{MR1200422}
Peter Greenberg.
\newblock Generators and relations in the classifying space for {PL}
  foliations.
\newblock {\em Topology Appl.}, 48(3):185--205, 1992.

\bibitem[GRW10]{galatius2010monoids}
S{\o}ren Galatius and Oscar Randal-Williams.
\newblock Monoids of moduli spaces of manifolds.
\newblock {\em Geom. Topol}, 14(3):1243--1302, 2010.

\bibitem[GS87]{ghys1987groupe}
{\'E}tienne Ghys and Vlad Sergiescu.
\newblock Sur un groupe remarquable de diff{\'e}omorphismes du cercle.
\newblock {\em Commentarii Mathematici Helvetici}, 62(1):185--239, 1987.

\bibitem[Hae70]{haefliger1970feuilletages}
Andr{\'e} Haefliger.
\newblock Feuilletages sur les vari{\'e}t{\'e}s ouvertes.
\newblock {\em Topology}, 9(2):183--194, 1970.

\bibitem[Hae71]{haefliger1971homotopy}
Andr{\'e} Haefliger.
\newblock Homotopy and integrability.
\newblock In {\em Manifolds-Amsterdam 1970}, pages 133--163. Springer, 1971.

\bibitem[HP64]{haefliger1964classification}
Andr{\'e} Haefliger and Valentin Poenaru.
\newblock La classification des immersions combinatories.
\newblock {\em Publications Math{\'e}matiques de l'Institut des Hautes
  {\'E}tudes Scientifiques}, 23(1):75--91, 1964.

\bibitem[KL66]{kuiper1966microbundles}
Nicolaas~H Kuiper and Richard~K Lashof.
\newblock Microbundles and bundles: I. elementary theory.
\newblock {\em Inventiones mathematicae}, 1(1):1--17, 1966.

\bibitem[KM05]{kotschick2005signatures}
D~Kotschick and S~Morita.
\newblock Signatures of foliated surface bundles and the symplectomorphism
  groups of surfaces.
\newblock {\em Topology}, 44(1):131--149, 2005.

\bibitem[Lan94]{MR1271828}
R\'{e}mi Langevin.
\newblock A list of questions about foliations.
\newblock In {\em Differential topology, foliations, and group actions ({R}io
  de {J}aneiro, 1992)}, volume 161 of {\em Contemp. Math.}, pages 59--80. Amer.
  Math. Soc., Providence, RI, 1994.

\bibitem[Mat71]{MR0288777}
John~N. Mather.
\newblock The vanishing of the homology of certain groups of homeomorphisms.
\newblock {\em Topology}, 10:297--298, 1971.

\bibitem[Mat74]{MR0356129}
John~N. Mather.
\newblock Commutators of diffeomorphisms.
\newblock {\em Comment. Math. Helv.}, 49:512--528, 1974.

\bibitem[Mat11]{mather2011homology}
John~N Mather.
\newblock On the homology of {H}aefliger's classifying space.
\newblock In {\em Differential Topology}, pages 71--116. Springer, 2011.

\bibitem[McD80]{mcduff1980homology}
Dusa McDuff.
\newblock The homology of some groups of diffeomorphisms.
\newblock {\em Commentarii Mathematici Helvetici}, 55(1):97--129, 1980.

\bibitem[Mei21]{meigniez2021quasicomplementary}
Ga{\"e}l Meigniez.
\newblock Quasicomplementary foliations and the {M}ather--{T}hurston theorem.
\newblock {\em Geometry \& Topology}, 25(2):643--710, 2021.

\bibitem[Mil58]{MR0095518}
John Milnor.
\newblock On the existence of a connection with curvature zero.
\newblock {\em Comment. Math. Helv.}, 32:215--223, 1958.

\bibitem[Mil83]{milnor1983homology}
John Milnor.
\newblock On the homology of lie groups made discrete.
\newblock {\em Commentarii Mathematici Helvetici}, 58(1):72--85, 1983.

\bibitem[Moi77]{MR0488059}
Edwin~E. Moise.
\newblock {\em Geometric topology in dimensions {$2$} and {$3$}}.
\newblock Graduate Texts in Mathematics, Vol. 47. Springer-Verlag, New
  York-Heidelberg, 1977.

\bibitem[Mor87]{morita1987characteristic}
Shigeyuki Morita.
\newblock Characteristic classes of surface bundles.
\newblock {\em Inventiones mathematicae}, 90(3):551--577, 1987.

\bibitem[MW07]{madsen2007stable}
Ib~Madsen and Michael Weiss.
\newblock The stable moduli space of {R}iemann surfaces: Mumford's conjecture.
\newblock {\em Annals of mathematics}, pages 843--941, 2007.

\bibitem[Nar17]{nariman2015stable}
Sam Nariman.
\newblock Stable homology of surface diffeomorphism groups made discrete.
\newblock {\em Geom. Topol.}, 21(5):3047--3092, 2017.

\bibitem[Nar23]{nariman2020thurston}
Sam Nariman.
\newblock Thurston's fragmentation and c-principles.
\newblock {\em to appear in {F}orum of {M}athematics, {S}igma.}, 2023.

\bibitem[NS90]{nesterenko1990homology}
Yu~P Nesterenko and Andrei~A Suslin.
\newblock Homology of the full linear group over a local ring, and milnor's
  k-theory.
\newblock {\em Mathematics of the USSR-Izvestiya}, 34(1):121, 1990.

\bibitem[PS83]{MR722372}
Walter Parry and Chih-Han Sah.
\newblock Third homology of {${\rm SL}(2,\,{\bf R})$} made discrete.
\newblock {\em J. Pure Appl. Algebra}, 30(2):181--209, 1983.

\bibitem[RW16]{randal2009resolutions}
Oscar Randal-Williams.
\newblock Resolutions of moduli spaces and homological stability.
\newblock {\em Journal of the European Mathematical Society}, 18:1--81, 2016.

\bibitem[Sus84]{MR772065}
Andrei~A. Suslin.
\newblock On the {$K$}-theory of local fields.
\newblock In {\em Proceedings of the {L}uminy conference on algebraic
  {$K$}-theory ({L}uminy, 1983)}, volume~34, pages 301--318, 1984.

\bibitem[Sus91]{Suslin}
Andrei Suslin.
\newblock K3 of a field and the bloch group.
\newblock In {\em Proceedings of the Steklov Math. Institute}, 1991.
\newblock Proceedings of the Steklov Math. Institute ; Conference date:
  01-01-1991.

\bibitem[Thu74a]{thurston1974foliations}
William Thurston.
\newblock Foliations and groups of diffeomorphisms.
\newblock {\em Bulletin of the American Mathematical Society}, 80(2):304--307,
  1974.

\bibitem[Thu74b]{MR0370619}
William Thurston.
\newblock The theory of foliations of codimension greater than one.
\newblock {\em Comment. Math. Helv.}, 49:214--231, 1974.

\bibitem[Thu74c]{thurston1974generalization}
William~P Thurston.
\newblock A generalization of the {R}eeb stability theorem.
\newblock {\em Topology}, 13(4):347--352, 1974.

\bibitem[Thu76]{MR0425985}
W.~P. Thurston.
\newblock Existence of codimension-one foliations.
\newblock {\em Ann. of Math. (2)}, 104(2):249--268, 1976.

\bibitem[Tsu89]{tsuboi1989foliated}
Takashi Tsuboi.
\newblock On the foliated products of class {$C ^1$}.
\newblock {\em Ann. Math.}, 130:227--271, 1989.

\bibitem[Tsu09]{tsuboi2009classifying}
Takashi Tsuboi.
\newblock Classifying spaces for groupoid structures.
\newblock {\em Contemporary Mathematics}, 498:67, 2009.

\bibitem[Wei13]{weibel2013k}
Charles~A Weibel.
\newblock {\em The K-book: An introduction to algebraic K-theory}, volume 145.
\newblock American Mathematical Society Providence, RI, 2013.

\end{thebibliography}
\end{document}